\documentclass[10pt,reqno]{amsart}

\usepackage[a4paper,left=35mm,right=35mm,top=30mm,bottom=30mm,marginpar=25mm]{geometry}

\usepackage{amsmath}
\usepackage{amssymb}
\usepackage{graphicx}
\usepackage{xcolor}
\usepackage{dsfont}
\usepackage{hyperref}
\usepackage{ifthen}

\makeindex

\newcommand{\R}{\mathbb{R}}
\newcommand{\N}{\mathbb{N}}
\newcommand{\Z}{\mathbb{Z}}
\newcommand{\V}{\mathcal{V}}

\newcommand{\T}{\mathbb{T}}

\newcommand{\cW}{\mathcal{W}}
\newcommand{\J}{\mathcal{J}}

\newcommand{\xbf}{\mathbf{x}}
\newcommand{\ybf}{\mathbf{y}}
\newcommand{\zbf}{\mathbf{z}}
\newcommand{\ubf}{\mathbf{u}}
\newcommand{\vbf}{\mathbf{v}}
\newcommand{\wbf}{\mathbf{w}}
\renewcommand{\epsilon}{\varepsilon}

\def\avint{\mathop{\,\,\rlap{\bf{--}}\!\!\int}\nolimits}

\numberwithin{equation}{section}

\newtheoremstyle{thmlemcorr*}{10pt}{10pt}{\itshape}{}{\bfseries}{.}\newline{{\thmname{#1}\thmnumber{
#2}\thmnote{ (#3)}}}
\newtheoremstyle{defi}{10pt}{10pt}{\itshape}{}{\bfseries}{.}{10pt}{{\thmname{#1}\thmnumber{
#2}\thmnote{ (#3)}}}
\newtheoremstyle{remexample}{10pt}{10pt}{}{}{\bfseries}{.}{10pt}{{\thmname{#1}\thmnumber{
#2}\thmnote{ (#3)}}}
\newtheoremstyle{ass}{10pt}{10pt}{}{}{\bfseries}{.}{10pt}{{\thmname{#1}\thmnumber{
A#2}\thmnote{ (#3)}}}

\newtheorem{thm}{Theorem}[section]

\newtheorem{lem}[thm]{Lemma}
\newtheorem{pro}[thm]{Proposition}

\newtheorem{prob}[thm]{Problem}

\theoremstyle{remark}
\newtheorem{rem}{Remark}[section]

\theoremstyle{definition}

\begin{document}

\title[Particle approximation of one-dimensional MFGs]{Particle approximation of one-dimensional Mean-Field-Games with local interactions}

\author[M. Di Francdesco]{Marco Di Francesco}
\address[Marco Di Francesco]{University of L'Aquila, Department of Information Engineering, Computer Science, and Mathematics (DISIM), Via Vetoio 1, Coppito, I-67100 L'Aquila, Italy.}
\email{marco.difrancesco@univaq.it}

\author[S. Duisembay]{Serikbolsyn Duisembay}
\address[S. Duisembay]{
        King Abdullah University of Science and Technology (KAUST), CEMSE Division , Thuwal 23955-6900. Saudi Arabia, and
        KAUST SRI, Center for Uncertainty Quantification in Computational Science and Engineering.}
\email{serikbolsyn.duisembay@kaust.edu.sa}

\author[D. A. Gomes]{Diogo Aguiar Gomes}
\address[D. A. Gomes]{
        King Abdullah University of Science and Technology (KAUST), CEMSE Division , Thuwal 23955-6900. Saudi Arabia, and
        KAUST SRI, Center for Uncertainty Quantification in Computational Science and Engineering.}
\email{diogo.gomes@kaust.edu.sa}

\author[R. Ribeiro]{Ricardo Ribeiro}
\address[R. Ribeiro]{
        King Abdullah University of Science and Technology (KAUST), CEMSE Division , Thuwal 23955-6900. Saudi Arabia, and
        KAUST SRI, Center for Uncertainty Quantification in Computational Science and Engineering.}
\email{ricardo.ribeiro@kaust.edu.sa}


\date{\today}

\begin{abstract}
We study a particle approximation for one-dimensional first-order Mean-Field-Games (MFGs) with local interactions with planning conditions.
Our problem comprises a system of a Hamilton-Jacobi equation coupled with a transport equation.
As we deal with the planning problem, we prescribe initial and terminal distributions for the transport equation.
The particle approximation builds on a semi-discrete variational problem.
First, we address the existence and uniqueness of a solution to the semi-discrete variational problem.
Next, we show that our discretization preserves some previously identified conserved quantities.
Finally, we prove that the approximation by particle systems preserves displacement convexity.
We use this last property to establish uniform estimates for the discrete problem.
We illustrate our results for the discrete problem  with numerical examples.

\end{abstract}

\maketitle

\section{Introduction}

Mean-field game theory is the study of the limiting behavior of systems comprising many identical rational agents.
In these models, rationality means that agents seek to minimize a cost functional.
Thus, each agent behaves as if it plays a dynamical game  whose information structure depends on its state and statistical information about the other agents.
This modeling framework was introduced in  \cite{ll1, ll2, ll3}, and around the same time, independently formulated in \cite{Caines2, Caines1}.
Here, we study the approximation of one-dimensional MFGs with local interactions by particle systems.


In classical MFGs, rational agents determine their optimal trajectories depending on their initial distribution and a terminal cost.
Here, we are interested in the planning problem for one-dimensional first-order MFGs.
These MFGs comprise a Hamilton-Jacobi equation coupled with a transport equation.
For the transport equation, we prescribe initial and terminal distributions.
Hence, in this problem, a terminal cost for the value function is not fixed; instead, it is chosen to steer the agents from the initial into the terminal configuration.

Let $\Omega$ be a spatial domain of the agents' positions, where $\Omega = \R$ or $\Omega = \T$, the standard unit torus identified with $\R / \Z$, and $T > 0$ be a fixed terminal time.
The planning problem that we consider is the following.

\begin{prob}\label{prob:mfg} Given a Hamiltonian $H: \R \to \R$, a potential $V: \Omega \times [0,T] \to \R$, a local coupling term $g: \R_0^+ \to \R$, an initial and a terminal distribution of agents $m_0, m_T:  \Omega \times [0,T] \to \R_0^+$.
Find $u: \Omega \times [0,T] \to \R$ and $m: \Omega \times [0,T] \to \R_0^+$ solving
\begin{equation}\label{eq:mfg}
\begin{cases}
-u_t + H(u_x) + V(x,t)= g(m), \\
m_t - (mH'(u_x))_x=0,
\end{cases}
\end{equation}
with
\[
m(\cdot,0) = m_0(\cdot)
\]
and
\[
m(\cdot,T) = m_T(\cdot).
\]
Here, $m_0, m_T \ge 0$ and $\int_\Omega m_0 dx = \int_\Omega m_T dx = 1$.
\end{prob}

In his lectures on MFGs \cite{cursolionsplanning}, P.-L.
Lions proved the existence and uniqueness of smooth solutions for Problem \ref{prob:mfg} and its second-order extension for quadratic Hamiltonians.
In \cite{porretta}, Porretta showed the existence of weak solutions for more general second-order planning problems.
The existence and uniqueness of the weak solutions were studied in \cite{Tono2019} and \cite{OrPoSa2018} via a variational approach.
A priori uniform estimates for the planning problem with potential were addressed in \cite{BakaryanFerreiraGomes2020}.
Finally, in \cite{LavSantambrogio2017}, the authors obtained $L^\infty$ bounds for  $m$ without a potential using a flow interchange technique.

While the above works address existence, uniqueness, and regularity questions, it is often hard to find explicit solutions to Problem \ref{prob:mfg}.
Therefore, numerical approximations are of paramount importance to understand the behavior of  solutions to Problem \ref{prob:mfg}.
There are several numerical methods to approximate MFGs; see
\cite{achdou2013finite} for a detailed account.
For example, 
in \cite{CDY}, the authors proposed a semi-implicit scheme for the optimal planning problem.
Also, they showed that the scheme preserves the existence and uniqueness of solutions for the discrete control problem.
The convergence of a finite-difference scheme for MFGs was studied in \cite{DY, CDY, Achdou2015}, and iterative strategies were investigated in \cite{MR2928376}.
Recently, several methods that can be traced
back to earlier works on optimal transport were considered in \cite{AL16II} and
\cite{briceno2018implementation}, and \cite{MR3772008}.
A distinct approach relies on the monotonicity structure that many MFGs 
share.
This monotonicity structure that is at the heart of the uniqueness proof 
by Lasry and Lions, provides an effective way to numerically approximate
MFGs.
Monotonicity methods were first introduced in \cite{almulla2017two} and later extended for time-dependent MFGs in \cite{GJS2}.
A new class of methods that combine ideas from monotone operators and Hessian Riemannian flows was recently developed in \cite{gomes2018hessian}.

Our approach is fundamentally distinct from the previous ones and uses
a \emph{Lagrangian particle} method.
This method was inspired
by earlier works on particle approximations for conservation laws in
\cite{MR3644595,MR3691810,MR3828235,MR3906267,MR4026959}.
See also previous results on linear and nonlinear diffusion equations \cite{MR1059326, MR2206449}.
We first give a variational representation of Problem \ref{prob:mfg}, then rewrite it in terms of the quantile function of the optimal density $m$.
Finally, we approximate the resulting problem via finite differences.

Our main assumptions read as follows.
As for the function $g$ in \eqref{eq:mfg}, we require the standing assumption
\[g\in C^1([0,+\infty)).\]
Moreover, we require
\begin{itemize}
	\item[(A1)]
	$(0,+\infty)\ni r\mapsto r g'(r)$ is locally integrable near $0$.
\end{itemize}
According to the preceding assumption, {the potential energy}
$G: \R^+ \rightarrow \R$, $G\in C^1((0,+\infty))$ given by the relation
\begin{equation}\label{eq:enthalpy}
  G'(r)=\frac{1}{r^2}\int_0^r s g'(s) ds
\end{equation}
is well defined (up to a constant).

As is customary in the context of MFGs, we assume the Hamiltonian function $H$ in \eqref{eq:mfg} to be convex.
{Hence,} we introduce the Legendre transform of $H$, given by
\[
L(v) = \sup_{p\in \R}\Big[-pv - H(p)\Big].
\]

We sketch our approximating procedure here. As a first step, 
as we explain in detail in Section \ref{sec:formal_comp}, 
Problem \ref{prob:mfg} can be formally recovered from the following variational problem.

\begin{prob}\label{prob:cont_min} Let $L: \R \to \R$ be a continuous function bounded by below, $G: \R_0^+ \rightarrow \R$ defined by \eqref{eq:enthalpy}, and $V \in C^1(\Omega \times [0,T],\R)$  continuous. Find $v: \Omega \times [0,T] \to \R$ and $m: \Omega \times [0,T] \to \R_0^+$ minimizing the functional
\[
\J(v,m) = \int_{\Omega} \int_{0}^T (L(v(s))+G(m(x,s))-V(x,s) ) m(x,s) ds dx,
\]
satisfying the differential constraint $m_t+(mv)_x=0$ and the initial/terminal conditions $m(\cdot,0) = m_0(\cdot)$ and $m(\cdot,T) = m_T(\cdot)$ in $\Omega$.
\end{prob}

We proceed to approximate the optimal density $m$ in Problem \ref{prob:cont_min} by the empirical measure of $N$ ordered particles modeling identical rational agents.
The states of these agents are given by the vector
\[\xbf(t) = (x_1(t),\dots,x_N(t)) \in \mathcal{K}^N,\]
where, if $\Omega = \R$,
\[\mathcal{K}^N:=\left\{\xbf=(x_1,\ldots,x_N)\in {\R^N}\,:\,\, x_1\leq \ldots \leq x_N\right\}\,,\]
and if $\Omega = \T$,
\[
\mathcal{K}^N:=\left\{\xbf=(x_1,\ldots,x_N)\in {\T^N}\,:\,\, x_1\leq \ldots \leq x_N \le x_1+1\right\}\,,
\] 
$0 \le t \le T$ and $N \in \N$. 
Let $\cW = L^\infty([0,T],\R^N)$ be the control set.
{Agents change their state by choosing a (time-dependent) control in $\cW$.}
For each control $\vbf \in \cW$, the states evolve according to
\[
\dot \xbf(t) = \vbf(t).
\]
The above equation is the discrete Lagrangian counterpart of the continuity equation $m_t+(mv)_x = 0$ in Problem \ref{prob:cont_min}.
More precisely, we think of the moving particles $\xbf(t) = (x_1(t),\dots,x_N(t)) \in \mathcal{K}^N$ as the quantiles of a time-dependent density, $m(\cdot,t)$.
Assuming, a priori, that the particles never touch each other, the following discrete $N$-dependent density is well defined
\begin{equation}\label{eq:rho_density}
\rho_N(x,t)=\sum_{i=1}^{N}R_i(t)\mathbf{1}_{[x_{i-1}(t),x_{i}(t))}(x)\,,
\end{equation}
where
\begin{equation}\label{eq:discrete_density}
R_i(t) = \frac{\delta_N}{x_i(t)-x_{i-1}(t)}, \qquad t \in [0,T],
\end{equation}
and $\delta_N = \frac{1}{N}$.
Notice that each of the agents carries a mass $\delta_N$.
The initial/terminal condition in the discrete setting is provided by
\begin{equation*}
  \mathbf{x}(0)=\xbf^0 \in \mathcal{K}^N\,,\qquad \mathbf{x}(T)=\xbf^T \in \mathcal{K}^N.
\end{equation*}
Each agent seeks to minimize a functional among all possible controls $\vbf\in\cW$.
Such functional is a proper discrete counterpart of $\J(v,m)$ in Problem \ref{prob:cont_min}, in which the density $m$ is replaced by \eqref{eq:rho_density}:
\[\widetilde \J (\xbf, \vbf) = \frac{1}{N} \sum_{i=1}^N \int_0^T \left[ L(v_i(s)) + G\left(\frac{\delta_N}{x_i(s)-x_{i-1}(s)}\right) - V(x_i(s),s) \right] ds\]
with $\vbf=\dot \xbf$.
We summarize the above in the formulation of our discrete optimization problem:
\begin{prob}\label{prob:global} Consider the setting of Problem \ref{prob:cont_min}. For a given $N \in \N$, find absolutely continuous trajectories $\xbf = (x_1,x_2,\dots,x_N): [0,T] \to \mathcal{K}^N$ minimizing the discrete utility functional
\begin{equation}\label{eq:global_problem}
\widetilde \J (\xbf, \dot \xbf) = \frac{1}{N} \sum_{i=1}^N \int_0^T \left[ L(\dot x_i(s)) + G\left(\frac{\delta_N}{x_i(s)-x_{i-1}(s)}\right) - V(x_i(s),s) \right] ds,
\end{equation}
with the initial-terminal condition
\[
  \mathbf{x}(0)=\xbf^0 \in \mathcal{K}^N\,,\qquad \mathbf{x}(T)=\xbf^T \in \mathcal{K}^N.
\]
Here, $\delta_N = \frac{1}{N}$, $x_0= x_N-1$ if $\Omega = \T$ and $x_0 = -\infty$ if $\Omega = \R$.
\end{prob}

{Problem \ref{prob:global} may be split in terms of individual functionals for each agent:}

\begin{prob}\label{problem} Consider the setting of Problem \ref{prob:cont_min}.
For all integers $i$ such that $1 \le i \le N$, find a smooth trajectory $x_i: [0,T] \to \R$  minimizing the functional, $J(x_{i-1},x_i, x_{i+1}, \dot x_i) $, given by 
\begin{equation}\label{eq:problem}
J= \frac{1}{N} \int_0^T \left[ L(\dot x_i(s))+  G\left(\frac{\delta_N}{x_i(s)-x_{i-1}(s)}\right) + G\left(\frac{\delta_N}{x_{i+1}(s)-x_{i}(s)}\right) - V(x_i(s),s) \right]ds
\end{equation}
with initial, $x_i(0)=x_i^0$, and terminal, $x_i(T) = x_i^T$, states.
Here, $\delta_N = \frac{1}{N}$ and $x_0= x_N-1$, $x_{N+1} = x_1+1$ if $\Omega = \T$ and $x_0 = -\infty$, $x_{N+1} = +\infty$ if $\Omega = \R$.
\end{prob}

The functional in \eqref{eq:problem} is identical for all $N$ agents.
Each agent seeks to determine an optimal trajectory given the trajectories of the other agents.
Thus, a solution to Problem \ref{problem} is a Nash equilibrium.


The Euler-Lagrange equations for \eqref{eq:global_problem} and \eqref{eq:problem} agree and are given by
\begin{equation}\label{eq:opt_condition}
L''(\dot x_i(t))\ddot x_i(t)= N (G'(R_{i+1}(t))R_{i+1}^2(t)- G'(R_{i}(t))R_{i}^2(t)) - V_x(x_i(t),t)\,,\qquad i=1,\ldots,N.
\end{equation}
Thus, this system of ODEs provides optimality conditions for both Problem \ref{prob:global} and Problem \ref{problem}, and, hence, their equivalence.
Moreover, spatial states, $\xbf(t) = (x_1(t),...,x_N(t))$, of the agents formally represent quantiles of the discrete density \eqref{eq:rho_density} at time $t$.

We provide a formal discussion on the discrete-to-continuum limit and the derivation of the continuum and discrete optimality conditions in Section \ref{sec:formal_comp}.
In particular, we formally show that the set of minimizers of the functional \eqref{eq:global_problem} are optimal trajectories of the agents, whose distribution function, $m$, solves Problem \ref{prob:mfg} as $N \to \infty$.

In Section \ref{sec:N_agent}, we prove the existence and uniqueness of a minimizer to the functional, $\widetilde\J$, in \eqref{eq:global_problem}.
For that, we need the following assumption.
\begin{itemize}
	\item[(A2)]
	$g \in C^1(\R_0^+, \R)$ and $g$ is non-decreasing.
\end{itemize}
We assume further that
\begin{itemize}
	\item[(A3)]
	$G:\R^+ \to \R$ is convex.
\end{itemize}

Note that, for some cases, $(A3)$ is a consequence of $(A2)$.
For example, if $g$ is convex, then $G$ is also convex.


To prove the uniqueness of a minimizer for \eqref{eq:global_problem}, we need two additional convexity assumptions:
\begin{itemize}
	\item[(A4)]
	The map $x \mapsto V(x,t)$ is concave for all $t$ such that $t \in [0,T]$.
\end{itemize}
Note that $x \mapsto V(x,t)$ is concave in $\Omega = \T$ only if it is constant.
\begin{itemize}
	\item[(A5)]
	The map $u \mapsto L(u)$ is uniformly convex; that is, there exists $\theta > 0$ such that for all $u,v \in \R$ and $0 \le \lambda \le 1$, we have
\[
L(\lambda u + (1-\lambda)v) \le \lambda L(u) + (1-\lambda)L(v) - \theta \lambda(1-\lambda) (u-v)^2.
\]
\end{itemize}




Let $\mathcal{P}(\Omega)$ be a space of probability measures and let $(u,m)$ solve Problem \ref{prob:mfg}.
A functional $\mathcal{U}: \mathcal{P}(\Omega) \to \R$ is \textit{displacement convex} with respect to Problem \ref{prob:mfg} if $t \mapsto \mathcal{U}(m(x,t))$ is convex.
Displacement convexity was first introduced in \cite{mccann1997convexity} to explore a non-convex variational problem.
In \cite{Schachter2018}, a new class of displacement convex functionals, which depend on the spatial derivatives of the density, was discovered for an optimal transport problem, where the system \eqref{eq:mfg} has no coupling ($g \equiv 0$).
In \cite{gomes2018displacement}, for the case $V \equiv 0$, authors identified a class of functions $U: \R_0^+ \to \R$ that satisfy the property
\begin{equation}\label{eq:conv_gs}
t \mapsto \int_\Omega U(m(x,t))dx \quad \text{is convex}
\end{equation}
for solutions of Problem \ref{prob:mfg} in $d$-dimension, $d \in \N$.
To be specific, for any convex function $U: \R_0^+ \to \R$, the property \eqref{eq:conv_gs} holds when $m$ solves Problem \ref{prob:mfg}.
The property \eqref{eq:conv_gs} also holds if the potential, $V$, is small enough (see  \cite{BakaryanFerreiraGomes2020}).
The displacement convexity is interesting because it provides a priori bounds for the density, $m$, which solves \eqref{eq:mfg}.
In this paper, we show that the preceding property holds at a discrete level as well.
In particular, we prove that for any convex function  $U: \R_0^+ \to \R$, we have
\begin{equation}\label{eq:disp_conv}
t \mapsto \sum_{i=1}^{N} U(R_i(t))(x_i(t) - x_{i-1}(t)) \quad \text{is convex},
\end{equation}
where $x_i(t)$ for $1 \le i \le N$ is a minimizer of \eqref{eq:global_problem} evaluated at time $t$ and $R_i(t)$ is defined by \eqref{eq:discrete_density}.
Because initial and terminal states of the agents are given,
the preceding convexity property implies the following a priori bound along the optimal trajectories:
\[
\begin{aligned}
\sum_{i=1}^{N} & U(R_i(t))(x_i(t) - x_{i-1}(t))\\
& \le \frac{t}{T} \sum_{i=1}^{N} U(R_{i}(T))(x_i(T) - x_{i-1}(T)) + \left(1 - \frac{t}{T}\right) \sum_{i=1}^N U(R_i(0))(x_i(0) - x_{i-1}(0)).
\end{aligned}
\]
In section \ref{sec:disp_conv}, we prove the following theorem and additional log-convexity property.
\begin{thm} \label{thm:disp_conv}  \textit{Let $\xbf(t) = (x_1(t),\dots,x_N(t)) \in C^1([0,T]\,;\,\Omega)$ solve Problem \ref{prob:global} for the case $V_x = 0$.
If $U: \R_0^+ \to \R$ is convex, then the map
\[
t \mapsto \sum_{i=1}^{N} U(R_i(t)) (x_i(t) - x_{i-1}(t))
\]
is convex, where $R_i(t)$ is defined by \eqref{eq:discrete_density}.}
\end{thm}
The above theorem implies that if the initial and terminal density of the agents is in $L^p$, then $m(x,t) \in L^p$ for all $t$ such that $t \in [0,T]$.
Consequently, it may be used to prove the weak convergence of discrete solutions towards solutions of the continuum variational problem \ref{prob:cont_min}, as we show in Section \ref{sec:estimates}.
Moreover, the displacement convexity property shows that there are no collisions between the agents.
The paper ends with numerical simulations provided in Section \ref{numerics}.
There, we show that the particle method provides a good approximation for the cumulative distribution function (CDF) of the density function, $m$, solving Problem \ref{prob:mfg}.

\section{Formal discrete-to-continuum limit and optimality conditions}\label{sec:formal_comp}

This section provides the formal link between the continuum minimization in Problem \ref{prob:cont_min} and its discrete counterpart in Problem \ref{prob:global}.
For both problems, we provide a formal derivation of the optimality conditions. To simplify the presentation, we only consider the case $\Omega=\R$.

\subsection{From the continuum variational problem to its formulation in the pseudo-inverse CDF}

Consider 
 a
pair $(m,v)$ 
 solving 
Problem \ref{prob:cont_min}
and assume 
that $\int_\Omega m(x,t) dx = 1$ for all $t$ such that $t\in [0,T]$.
The CDF, $F:\Omega\times [0,T]\to [0,1]$, is
\begin{equation}\label{eq:cdf}
F(x,t)=\int_{-\infty}^x m(y,t) dy.
\end{equation}
Because $F$ is non-decreasing in $x$, we define its pseudo-inverse, $X:[0,1]\times [0,T]\rightarrow \mathbb{R}$, as
\[X(z,t):=\inf\left\{x\in \R\,:\,\, F(x,t)\geq z\right\}\,.\]
The following formal computation requires enough regularity 
on the involved variables, $X$, $m$ and $F$, which, in principle, is not guaranteed.
First, we have
$X( F(x,t),t)\leq x$, with equality whenever $F_x(x,t) > 0$.
Therefore, 
\[1=\partial_x X( F(x,t),t)= X_z(F(x,t),t) F_x(x,t).\]
Accordingly, taking into account that $F_x(x,t)=m(x,t)$, we have
\begin{equation}\label{eq:change_1}
  m(x,t)=(X_z(F(x,t),t))^{-1}\,,
\end{equation}
provided $X_z(F(x,t),t)\neq 0$.
Moreover, the identity
\[0 =\partial_t X(F(x,t),t) = X_z(F(x,t),t) F_t(x,t) + X_t(F(x,t),t)\]
implies
\[X_t(F(x,t),t)=-X_z(F(x,t),t) F_t(x,t)\,.\]
Next, we 
integrate the continuity equation $m_t+(m v)_x = 0$ with respect to $x$ on $(-\infty,x]$ neglecting the boundary term at $-\infty$. 
Accordingly, multiplying the resulting expression by 
$X_z(F(x,t),t)$, 
we formally obtain
\begin{equation}\label{eq:change_2}
  X_t(F(x,t),t) =  X_z(F(x,t),t) m(x,t) v(x,t) = v(x,t)\,.
\end{equation}
Hence, by defining $\mathcal{V}:[0,1]\times [0,T]\rightarrow \R$ as
\[\mathcal{V}(z,t)=v(X(z,t),t)\,,\]
the continuity equation in Problem \ref{prob:cont_min} in the new variables $X(z,t), \mathcal{V}(z,t)$ with $z=F(x,t)$ becomes
\[X_t(z,t)=\mathcal{V}(z,t)\,.\]
Concerning the functional $\J(m,v)$, we perform the change of variable $z=F(x,t)$.
Accordingly,  $m dx = dz$.
We then obtain $\J(m,v)=\widetilde{\J}(X,\mathcal{V})$ with
\[\widetilde{\J}(X,\mathcal{V})=\int_0^T \int_{0}^1 \left(L(\mathcal{V}(z,s)) + G(X_z(z,s)^{-1}) -V(X(z,s),s)\right) dz ds\,.\]
Hence, Problem \ref{prob:cont_min} can be formally converted into the following problem.
\begin{prob}\label{prob:cont_min_2} Let $L: \R \to \R$ be a continuous function bounded by below, $G: \R_0^+ \rightarrow \R$ defined by \eqref{eq:enthalpy}, and $V \in C^1(\Omega \times [0,T],\R)$  continuous.
Find $\mathcal{V}:[0,1]\times [0,T]\rightarrow \Omega$ and $X: [0,1]\times [0,T] \to \Omega$ minimizing the functional
\[
\widetilde{\J}(X,
\mathcal{V}) =\int_0^T \int_{0}^1 \left(L(\mathcal{V}(z,s)) + G(X_z(z,s)^{-1}) -V(X(z,s),s)\right) dz ds\,,
\]
subject to the differential constraint $X_t(z,t)=\mathcal{V}(z,t)$ with $X(\cdot,0) = X_0(\cdot)$ and $X(\cdot,T) = X_T(\cdot)$ in $[0,1]$.
\end{prob}

\subsection{From the continuum CDF to the discrete variational problem}\label{subsec:pseudo}

We next derive Problem \ref{prob:global} as a discretization of Problem \ref{prob:cont_min_2}.
For a fixed $N\in \mathbb{N}$, we split the mass interval $(0,1]$ into $N$ subintervals of equal size $I_j=((j-1)/N,j/N]$, $j\in \{1,\ldots, N\}$.
Then, we approximate the variable $X(z,t)$ in two distinct ways, a piecewise constant one, $X_N$, and a piecewise linear one, $\tilde{X}_N$.
With the notation
\[
\begin{cases}
x_0(t)& =-\infty \\
x_j(t)&=X(j/N,t)\,,\qquad j\in \{1,\ldots,N\}\,,
\end{cases}
\]
where  $t\in [0,T]$, we set
\[
   X_N(z,t)=x_0(t)\mathbf{1}_{\{0\}}(z) + \sum_{j=1}^N x_{j}(t)\mathbf{1}_{((j-1)/N,j/N]}(z),
\]
and
\[
\widetilde{X}_N(z,t)=x_0(t)\mathbf{1}_{[0,1/N]}(z) +\sum_{j=2}^N (x_{j-1}(t) + N(x_j(t)-x_{j-1}(t))(z-(j-1)/N)\mathbf{1}_{((j-1)/N,j/N]}(z).
\]
We approximate the variable $\mathcal{V}(z,t)$ in the piecewise constant form
\[\mathcal{V}_N(z,t)=v_0(t)\mathbf{1}_{\{0\}}(z) + \sum_{j=1}^N v_{j}(t)\mathbf{1}_{((j-1)/N,j/N]}(z)\,
\]
for $t\in [0,T]$, with the notation
\[
\begin{cases}
v_0(t)&=0 \\
v_j(t)&=\mathcal{V}(j/N,t)\,,\qquad j\in \{1,\ldots,N\}\,.
\end{cases}
\]
The scaled continuity equation $X_t(z,t)=\mathcal{V}(z,t)$ in Problem \ref{prob:cont_min_2} becomes
\[\dot{x}_j(t)= v_j(t)\,.\]
We now choose to substitute the $X, \mathcal{V}$ terms in the functional $\widetilde{\J}$ in Problem \ref{prob:cont_min_2} as follows.
$X$ and $\mathcal{V}$ are replaced in the zero-order terms by $X_N$ and $\mathcal{V}_N$, respectively.
For $X_z^{-1}$, we use the piecewise linear interpolation $\widetilde{X}$ and obtain
\[
\left(\partial_z \widetilde{X}(z,t)\right)^{-1} =  \frac{1}{N(x_{j}(t)-x_{j-1}(t))}, \qquad z\in I_j \text{ for } j\in \{1,\ldots,N\}.
\]
The above substitutions turn the continuum functional $\widetilde{\J}(X,\mathcal{V})$ into the discrete functional $\widetilde{\J}$ in Problem \ref{prob:global}.

\subsection{Optimality conditions for the discrete variational problem}

Now, we derive the optimality conditions for Problem \ref{prob:global}.
Let $\xbf:[0,T]\rightarrow \mathcal{K}^N$ be an optimal trajectory for Problem \ref{prob:global}.
We consider a small perturbation $\xbf+\varepsilon \ybf$ with arbitrary $\ybf=(y_1,\ldots,y_N):[0,T]\rightarrow \mathcal{K}^N$ with compact support on $(0,T)$ and $\varepsilon << 1$.
Setting $x_{N+1}=+\infty$ and $y_0=y_N$, we get
\begin{align*}
   0  & = \left.\frac{d}{d\varepsilon}\right|_{\varepsilon=0} \widetilde{\J}(\xbf+\varepsilon\ybf,\dot \xbf+\varepsilon \dot\ybf) \\
    & = \delta_N\sum_{i=1}^N \int_0^T L'(\dot{x}_i(t))\dot{y}_i(t) dt \\
    & \ + \delta_N\sum_{i=1}^N \int_0^T G'\left(\frac{\delta_N}{x_i(t)-x_{i-1}(t)}\right) \left(-\frac{\delta_N}{(x_i(t)-x_{i-1}(t))^2}\right)(y_i(t)-y_{i-1}(t)) dt \\
    & \ -\delta_N \sum_{i=1}^N \int_0^T V_x(x_i(t),t) y_i(t) dt.
\end{align*}
Integration by parts w.r.t.
$t$ in the first term and summation by parts w.r.t.
$i$ in the second term above yields
\begin{align*}
   & 0= \delta_N\sum_{i=1}^{N} \int_{0}^T - L''(\dot{x}_i(t))\ddot{x_i}(t) y_i(t) dt \\
   & \ +\sum_{i=1}^{N}\int_0^T \left(G'(R_{i+1}(t))R_{i+1}^2(t) - G'(R_i(t))R_i^2(t)\right)y_i(t) dt\\
   & \ -\delta_N \sum_{i=1}^N \int_0^T V_x(x_i(t),t) y_i(t) dt\\
   & \ +\int_0^T\left(G'(R_1(t))R_1^2(t) -G'(R_{N+1}(t))R_{N+1}(t)^2\right) y_{N}(t)dt,
\end{align*}
where the last term above equals zero. Indeed, both $R_1$ and $R_{N+1}$ are zero because they are the reciprocal of $+\infty$.
Hence, due to the arbitrariness of the perturbation, the ODE system \eqref{eq:opt_condition} is satisfied.

\subsection{From discrete optimality conditions to the mean-field game}

Finally, we use the optimality conditions from the prior section to formally recover the mean-field game in Problem \ref{prob:mfg}.
Here, we show that
the system of partial differential equations in Problem \ref{prob:mfg} with initial and terminal conditions, $m_0$ and $m_T$, can be formally recovered as the Euler-Lagrange system for the minimization Problem \ref{prob:cont_min}.
An alternative way to recover the PDE system is the many-particle limit of the discrete optimality condition \eqref{eq:opt_condition}.
For that, we set $X_N$, $\widetilde{X}_N$, and $\V_N$ as in subsection \ref{subsec:pseudo}, where $x_1(t),\ldots,x_N(t)$ solve \eqref{eq:opt_condition}, and $v_i(t)=\dot{x}_i(t)$ for all $i=1,\ldots,N$.
Then, \eqref{eq:opt_condition} can be formally seen as a semi-discrete finite-difference approximation of the PDE system
\begin{align}
  & X_t(z,t)=\V(z,t)\label{eq:pseudo_continuity}\\
  & \partial_t L'(\V(z,t)) = \partial_z B(X_z(z,t)^{-1})-V_x(X(z,t),t),\label{eq:pseudo_HJ}
\end{align}
where
\[B(\rho):=G'(\rho)\rho^2.\]
Notice that we are implicitly assuming that $\widetilde{X}_N$ and $X_N$ approximate the same pseudo-inverse function $X$.
As in subsection \ref{subsec:pseudo}, we set
\[F(x,t):=\inf\left\{z\in [0,1]\,:\,\,X(z,t)\geq x\right\}\,,\]
and $m(x,t)=\partial_x F(x,t)$.
We also set
\[v(x,t):=\V(F(x,t),t)\,\]
and
\[u(x,t):=\int_{0}^x L'(v(y,t))dy\,.\]
We use the change of variable $x=X(z,t)$ as in subsection \ref{subsec:pseudo}.
From \eqref{eq:change_1} and \eqref{eq:change_2}, we obtain that \eqref{eq:pseudo_continuity} becomes
\[F_t(x,t)=-m(x,t) v(x,t),\]
which, upon differentiation w.r.t. $x$, gives
\begin{equation}\label{eq:continuity2}
  m_t(x,t)+(m(x,t)v(x,t))_x = 0\,.
\end{equation}
Since $u_x(x,t)=-L'(v(x,t))$, assuming strict convexity of $L$, we have 
\[v(x,t)=(L')^{-1}(-u_x(x,t)) = - H'(u_x(x,t))\,.\]
Therefore, the continuity equation \eqref{eq:continuity2} becomes
\[m_t(x,t)-(m(x,t)H'(u_x(x,t)))_x = 0\,,\]
which is consistent with the second equation in \eqref{eq:mfg}.
As for \eqref{eq:pseudo_HJ}, we observe
\begin{align*}
   & \partial_t L'(\V(z,t)) = - \partial_t u_x(X(z,t),t) = -u_{xt}(X(z,t),t) - u_{xx}(X(z,t),t)X_t(z,t)\\
   & \ = - u_{xt}(x,t) - u_{xx}(x,t)v(x,t) = - u_{xt}(x,t) + u_{xx}(x,t)H'(u_x(x,t))\\
   & \ = \partial_x \left(- u_t(x,t) + H(u_x(x,t))\right)
\end{align*}
and
\begin{align*}
  & \partial_z B(X_z(z,t)^{-1})-V_x(X(z,t),t) = B'(m(X(z,t),t)) 
  m_x(X(z,t),t)X_z(z,t) - V_x(X(z,t),t)\\
  & \ = \frac{B'(m(x,t))}{m(x,t)}m_x(x,t)- V_x(x,t), 
\end{align*}
for $x=X(z,t)$.
We compute
\[\frac{B'(m)}{m}=\frac{1}{m}\frac{d}{dm}(G'(m)m^2)= \frac{1}{m} m g'(m) = g'(m)\,.\]
Hence, we obtain
\[\partial_x \left(- u_t(x,t) + H(u_x(x,t))\right) = (g(m(x,t)))_x - V_x(x,t)\,,\]
which, upon integration with respect to $x$, gives
\[- u_t(x,t) + H(u_x(x,t)) = g(m(x,t)) - V(x,t)\,,\]
which coincides with the first PDE in \eqref{eq:mfg}.

\section{Existence and uniqueness results}\label{sec:N_agent}
In this section, we show the existence and uniqueness of a solution to Problem \ref{prob:global}.

\subsection{Existence and uniqueness of a solution}
In the following lemma, we show that the functional \eqref{eq:global_problem} is convex.
\begin{lem}
     Assume that 
     (A1)-(A5) hold.
Then, $(\xbf, \ubf) \mapsto \widetilde{\J}(\xbf, \ubf)$ is convex.
\end{lem}

\begin{proof}
Assumption (A5) immediately implies that the map $\ubf \mapsto \widetilde{\J}(\xbf, \ubf)$ is convex.
The convexity of the functional $\widetilde{\J}(\xbf, \ubf)$ in $\xbf$ follows from the convexity of the maps 
$x\mapsto -V(x,t)$, by Assumption $(A4)$, and $x \mapsto G\left(\frac{\delta_N}{x-a}\right)+G\left(\frac{\delta_N}{b-x}\right)$.
The latter is satisfied if
\[\frac{1}{x^4}G''\left(\frac{1}{x}\right) +\frac{2}{x^3}G'\left(\frac{1}{x}\right) \geq 0;\]
that is, with $\rho=1/x$,
\begin{equation*}
\rho G''(\rho) + 2 G'(\rho)\geq 0, \quad \text{ for } \rho > 0.
\end{equation*}
The definition of $G$ implies
\[\rho G''(\rho) + 2 G'(\rho) = g'(\rho) \ge 0.\]
Because $g$ is non-decreasing by (A2), the prior inequality is true.
\end{proof}

The preceding lemma is used to prove the existence and uniqueness of a solution for Problem \ref{prob:global}.
Let
\[
I(x_{i-1}(s), x_i(s),u_i(s)) =L(u_i(s))+ G\left( \frac{\delta_N}{x_i(s) - x_{i-1}(s)} \right)
\]
and
\[
\mathcal{L}(\xbf(s),\ubf(s),s) = \frac{1}{N} \sum_{i=1}^N  \left( I(x_{i-1}(s), x_i(s),u_i(s)) - V(x_i(s),s) \right).
\]
Then, \eqref{eq:global_problem} becomes
\[
\widetilde{\J}(\xbf,\ubf) = \int_0^T \mathcal{L}(\xbf(s),\ubf(s),s)ds.
\]

\begin{pro} Assume (A1)-(A5) and let $I(x,y,u) \ge \alpha |u|^q-\beta$ for some $q \in (1,\infty)$, $\alpha > 0$, $\beta \ge 0$, for every $x$ and $y$ such that $x<y$.
Then, there exists a unique solution $(\xbf,\ubf) \in C^1([0,T];\mathcal{K}^N) \times L^\infty([0,T];\R^N)$ of Problem \ref{prob:global} with $\dot \xbf = \ubf$.
\end{pro}
\begin{proof} First, we prove the existence. 
We use the direct method of the calculus of variations.
Because $V$ is concave by $(A4)$, there is a lower bound, $\gamma$, for $-V$.
Then,
\[
\begin{aligned}
\mathcal{L}(\xbf(t), \ubf(t)) & = \frac{1}{N} \sum_{i=1}^N  \left[I\left(x_{i-1}(t),x_i(t),u_i(t)\right) - V\left(x_i(t),t\right) \right]\\ 
&\ge \frac{1}{N} \sum_{i=1}^N \left[\alpha \left|u_i(t)\right|^q-\beta\right] + \gamma \\
& = \frac{\alpha}{N} \|\ubf(t)\|_{L^q}^q - \beta' \ge \frac{\alpha C}{N} |\ubf(t)|^q - \beta',
\end{aligned}
\]
where $\beta'=\beta - \gamma$ and $C$ is some constant.
The last inequality follows from the fact that all finite-dimensional $q$-norms are equivalent in $\R^N$.
Hence, $\mathcal{L}$ is coercive in $\ubf$.
The coercivity condition of $\mathcal{L}$ leads to its boundedness from below.
Also, from $(A5)$, it follows that $\mathcal{L}$ is convex on the second variable.
Thus, it is lower semicontinuous.

Define the admissible set by
\[
\mathcal{A} = \{\xbf \in W^{1,q}(0,T) \quad | \quad  \xbf(0) = \xbf^0, \xbf(T)=\xbf^T\}.
\]
The functional $\widetilde{\J}$ is bounded by below because $\mathcal{L}$ is bounded by below. Hence, we can find a minimizing sequence $(\xbf^n)_{n \in \N} \subset \mathcal{A}$ such that
\[
\lim_{n \to +\infty} \widetilde{\J}(\xbf^n,\dot \xbf^n) = \inf \widetilde{\J} (\xbf,\dot \xbf).
\]
By the coercivity of $\mathcal{L}$,
\[
\widetilde{\J}(\xbf^n,\dot \xbf^n) \ge \alpha' \|\dot \xbf^n\|_{L^q}^q - \beta' T,
\]
where $\alpha' = \frac{\alpha T}{N}$.
Thus, the sequence $(\xbf^n)_{n \in \N}$ is bounded in $W^{1,q}(0,T)$.
Consequently, we can find a subsequence, still denoted  $\xbf^n$, and a function
 $\xbf^* \in W^{1,q}(0,T)$ such that $\xbf^n$ weakly converges to $\xbf^*$.
 Because $q > 1$, it follows from Morrey's theorem (see \cite{E6}) that the set $\mathcal{A}$ is closed.
 Hence, by the convexity of $\mathcal{A}$, 
 Mazur's theorem (see \cite{E6}, Appendix D.4) gives that $\mathcal{A}$ is weakly closed in $W^{1,q}((0,T))$.
 Consequently,   $\xbf^* \in \mathcal{A}$.
 Then, because $\mathcal{L}$ is bounded by below and convex in $\ubf$, $\widetilde{\J}$ is lower semicontinuous.
 Therefore, $\xbf^*$ minimizes $\widetilde{\J}$ because
\[
\inf_\xbf \widetilde{\J}(\xbf, \dot \xbf) = \lim_{n \to +\infty} \widetilde{\J}(\xbf^n,\dot \xbf^n) \ge \widetilde{\J} (\xbf^*, \dot \xbf^*) \ge \inf_{\xbf} \widetilde{\J}(\xbf,\dot \xbf).
\]
Next, we prove the uniqueness.
Suppose that there are two minimizers $(\xbf,\ubf)$, $(\ybf,\vbf)$ such that $\dot \xbf = \ubf$, $\dot \ybf = \vbf$ and $\xbf(0) = \ybf(0) = \xbf^0$;
that is, 
\[
\min_{\substack{\zbf \in \Omega \\ \dot \zbf = \wbf \in \R}} \widetilde{\J} (\zbf,\wbf) = \widetilde{\J}(\xbf,\ubf) = \widetilde{\J}(\ybf,\vbf).
\]
Because $I(x,y,u)$ is convex in $x$ and $y$, and uniformly convex in $u$, and $-V(x,s)$ is convex in $x$, we have
\[
\begin{aligned}
\widetilde{\J} \left(\frac{\xbf+\ybf}{2},\frac{\ubf+\vbf}{2}\right) & = \frac{1}{N} \sum_{i=1}^{N} \left[ \int_0^T I\left(\frac{x_{i-1}+y_{i-1}}{2},\frac{x_i+y_i}{2}, \frac{u_i+v_i}{2}\right)ds - \int_0^T V\left(\frac{x_i+y_i}{2},s\right)ds \right] \\
&  \begin{aligned}  \le \frac{1}{2N} \sum_{i=1}^{N} \Bigg[ \int_0^T I(x_{i-1},x_i,u_i)ds & +  \int_0^T I(y_{i-1},y_i,v_i)ds - \frac{\theta}{2}(u_i-v_i)^2\\ &  -  \int_0^T V(x_i,s)ds -  \int_0^T V(y_i,s) ds  \Bigg] \end{aligned} \\
& = \frac{1}{2} \widetilde{\J} (\xbf,\ubf) +  \frac{1}{2} \widetilde{\J} (\ybf,\vbf) - \frac{\theta}{4N} \|\ubf-\vbf\|^2 \\
& = \min_{\zbf \in \mathcal{K}_N, \dot \zbf = \wbf} \widetilde{\J}(\zbf,\wbf) - \frac{\theta}{4N} \|\ubf-\vbf\|^2.
\end{aligned} 
\]
Because $\theta > 0$, the preceding inequality is possible only if $\ubf = \vbf$.
This means that $\xbf - \ybf = c$ for some $c \in \R^N$.
From the initial conditions, $\xbf(0) = \ybf(0)$, we obtain that $c=\textbf{0}\in \R^N$, and hence $\xbf = \ybf$.
\end{proof}

\section{Conserved quantities}\label{sec:cons_quan}

In \cite{MR3575617}, authors determined continuous conserved quantities, $E \in C^2(\R \times \R^+)$, for Problem \ref{prob:mfg}; that is quantities such that
\begin{equation}\label{eq:cont_cons_quan}
\frac{d}{dt} \int_{\T} E(v,m)dx = 0,
\end{equation}
where $v=u_x$ and the pair $(u,m) \in C^2(\T \times (0,\infty)) \cap C(\T \times [0,\infty))$ solves Problem \ref{prob:mfg} with $V$ such that $V_x = 0$. 
For example, the function $E(v,m) = \alpha v + \beta m$, $\alpha, \beta \in \R$, $E(v,m)=mv$, and $E(v,m)=H(v)-P(m)$, where $P(m)=\frac{g'(m)}{m}$, satisfy \eqref{eq:cont_cons_quan}.

In this section, we identify some conserved quantities for Problem \ref{prob:global} at a discrete level for the periodic case, $\Omega = \T$.
Thus, for all $i$ such that $1 \le i \le N$ and any solution $\xbf(t)=(x_1(t),\dots,x_N(t)) \in \mathcal{K}^N$ of Problem \ref{prob:global}, we seek functions $E^i(u_i,R_i) \in C^1(\R \times \R^+)$ satisfying
\begin{equation}\label{eq:cons_quant_def}
\frac{d}{dt} \sum_{i=1}^{N} E^i(u_i, R_i)(x_i-x_{i-1}) = 0,
\end{equation}
where $u_i=\dot x_i$ and $R_i$ is defined by \eqref{eq:discrete_density}.
We say that $E^i(u_i, R_i)$ is a \textit{semi-discrete conserved quantity} for Problem \ref{prob:global} if \eqref{eq:cons_quant_def} holds.
With
\begin{equation}\label{eq:E_P}
P(u_i,R_i) = \frac{ E^i(u_i, R_i)}{R_i},
\end{equation} 
because $x_i-x_{i-1} = \frac{\delta_N}{R_i}$, \eqref{eq:E_P} is equivalent to
\begin{equation}\label{eq:P}
\frac{d}{dt} \sum_{i=1}^{N} P^i(u_i, R_i)=0.
\end{equation}
Thus, using that $\dot R_i = - N R_i^2(u_i-u_{i-1})$, the optimality condition \eqref{eq:opt_condition}, and the periodicity, we get
\[
\begin{aligned}
\frac{d}{dt} & \sum_{i=1}^N P^i(u_i, R_i)  = \sum_{i=1}^N \frac{\partial P^i}{\partial u_i} \ddot x_i - \frac{\partial P^i}{\partial R_i} N R_i^2 (u_i-u_{i-1}) \\
& =  \sum_{i=1}^N \frac{\partial P^i}{\partial u_i} \frac{\left(N G'(R_{i+1})R_{i+1}^2 - N G'(R_i)R_i^2 - V_x(x_i,t)\right)}{L''(u_i)} - \frac{\partial P^i}{\partial R_i} N R_i^2 (u_i-u_{i-1}) \\
& = \sum_{i=1}^N \frac{\partial P^{i-1}}{\partial u_{i-1}}\frac{N G'(R_i) R_i^2}{L''(u_{i-1})} - \frac{\partial P^i}{\partial u_i} \frac{N G'(R_i) R_i^2}{L''(u_{i})} - \frac{\partial P^i}{\partial u_i} \frac{V_x(x_i,t)}{L''(u_i)} - \frac{\partial P^i}{\partial R_i} N R_i^2 (u_i-u_{i-1}) \\
& =  \sum_{i=1}^N \left[\left(\frac{\partial P^{i-1}}{\partial u_{i-1}} \frac{1}{L''(u_{i-1})}- \frac{\partial P^i}{\partial u_i}\frac{1}{L''(u_i)}\right)G'(R_i) - \frac{\partial P^i}{\partial R_i}(u_i-u_{i-1})\right] N R_i^2 - \frac{\partial P^i}{\partial u_i}\frac{V_x(x_i,t)}{L''(u_i)}.
\end{aligned}
\]
Hence, when $V$ is constant, the last term in the previous expression vanishes.
To solve \eqref{eq:P}, we look for a function $P^i(u_i,R_i)$ that satisfies
\[
\left(\frac{\partial P^{i-1}}{\partial u_{i-1}}\frac{1}{L''(u_{i-1})}- \frac{\partial P^i}{\partial u_i}\frac{1}{L''(u_{i})}\right)G'(R_i) - \frac{\partial P^i}{\partial R_i}(u_i-u_{i-1}) = 0.
\]
Next, we obtain some solutions of the preceding equation using separation of variables.
We fix $\lambda \in \R$ and look for solutions that satisfy
\[
\frac{\frac{\partial P^{i}}{\partial u_{i}} \frac{1}{L''(u_{i})}- \frac{\partial P^{i-1}}{\partial u_{i-1}}\frac{1}{L''(u_{i-1})}}{u_i-u_{i-1}} = -\frac{\partial P^i}{\partial R_i} \frac{1}{G'(R_i)} = \lambda.
\]
There are two cases, $\lambda=0$ or $\lambda \neq 0$. In the latter case, we take $\lambda=1$ without loss of generality.
If $\lambda=0$, then
\[
\frac{\partial P^{i}}{\partial u_{i}} \frac{1}{L''(u_{i})} = \frac{\partial P^{i-1}}{\partial u_{i-1}}\frac{1}{L''(u_{i-1})}
\]
and 
\[
\frac{\partial P^i}{\partial R_i} = 0.
\]
Hence, one possible solution of \eqref{eq:P} is $P^i(u_i,R_i) = L'(u_i)$.

For $\lambda = 1$, we get 
\[
\frac{\partial P^{i}}{\partial u_{i}}  =  L''(u_i) u_i
\]
and
\[
\frac{\partial P^i}{\partial R_i} = - G'(R_i).
\]
These equations yield $P^i(u_i,R_i) = L'(u_i)u_i-L(u_i)-G(R_i)$.
In conclusion, substituting in \eqref{eq:E_P}, we obtain that 
\begin{equation}\label{eq:cons_momentum}
E^i(u_i,R_i) = L'(u_i)R_i, \qquad (\lambda=0)
\end{equation}
and 
\begin{equation}\label{eq:cons_energy}
E^i(u_i,R_i)= (L'(u_i)u_i - L(u_i) - G(R_i))R_i, \qquad (\lambda=1)
\end{equation}
are semi-discrete conserved quantities.
These are the semi-discrete counterparts of the conserved quantities $E(m,v) = mv$ and $E(m,v) = \frac{mv^2}{2}-\frac{m^3}{6}$, respectively, for the case $L(v)=\frac{v^2}{2}$ and $G(r)=\frac{r^2}{6}$ determined in \cite{MR3575617} for \eqref{eq:mfg}, where $v=u_x$.

\section{Displacement convexity}\label{sec:disp_conv}
In this section, we prove Theorem \ref{thm:disp_conv} for $\Omega = \T$. Also, we provide upper bounds for the $L^p$ norms of the sequence $\{R_i(t)\}_{i=1}^N$ in terms of initial and terminal distributions, $\{R_i(0)\}_{i=1}^N$ and $\{R_i(T)\}_{i=1}^N$, for any $t$ such that $0 \le t \le T$.

\begin{proof}[Proof of Theorem \ref{thm:disp_conv}]
For simplicity of exposition in this proof, we drop the dependence of $x_i$ on $t$ and subscript of $\delta_N$.
Denoting $\xi_i = \frac{1}{R_i}= \frac{x_i-x_{i-1}}{\delta}$, we show that
\begin{equation}\label{eq:disp_conv_disc_U}
\frac{d^2}{dt^2} \sum_{i=1}^{N}  U\left(\frac{1}{\xi_i}\right) \xi_i \ge 0.
\end{equation}
Note that because any minimizer, $\xbf=(x_1,\dots,x_N):[0,T] \to \mathcal{K}^N$, of Problem \ref{prob:global} satisfies $x_i>x_{i-1}$, we have $\xi_i>0$ for all integer $i$ such that $1\le i \le N$.

We begin by computing the first derivative:
\[
\frac{d}{dt} \sum_{i=1}^{N} U\left(\frac{1}{\xi_i}\right) \xi_i = \sum_{i=1}^{N} -U'\left(\frac{1}{\xi_i}\right)\frac{\dot \xi_i}{\xi_i} + U\left(\frac{1}{\xi_i}\right)\dot \xi_i.
\]
By differentiating again, we obtain
\begin{equation}\label{eq:second_der_U}
\frac{d^2}{dt^2} \sum_{i=1}^{N}  U\left(\frac{1}{\xi_i}\right) \xi_i  = \underbrace{\sum_{i=1}^{N} U''\left(\frac{1}{\xi_i}\right)\frac{\dot \xi_i^2}{\xi_i^3}}_{\text{A}} + \underbrace{ \sum_{i=1}^{N}  \Bigg( U\left(\frac{1}{\xi_i}\right)  - U'\left(\frac{1}{\xi_i}\right)\left(\frac{1}{\xi_i}\right) \Bigg)\ddot \xi_i}_{B}.
\end{equation}
Because of the convexity of $U$ and that $\xi_i>0$, the term (A) is non-negative.
Hence, it is enough to show that $(B)$ is non-negative

Let $P(z)=U'(z)z-U(z)$.
Note that because $U$ is convex, $P'(z)=U''(z)z \ge 0$ for $z \ge 0$, and hence, 
\begin{equation}\label{eq:P_nondecreasing}
P(z)  \text{ is non-decreasing for $z \ge 0$.}
\end{equation}
Also, from the optimality condition \eqref{eq:opt_condition}, for a constant potential $V$, we have
\begin{equation}\label{eq:EL_discrete}
\delta \ddot x_i  = \Big(L''(\dot x_i)\Big)^{-1}\left( G'\left(\frac{1}{\xi_{i+1}}\right)\frac{1}{\xi_{i+1}^2}-G'\left(\frac{1}{\xi_i}\right) \frac{1}{\xi_i^2}\right).
\end{equation}

We rewrite the term $(B)$ in \eqref{eq:second_der_U}, multiplied by $\delta$, using \eqref{eq:EL_discrete} and the definition of $P(z)$:
\[
\begin{aligned}
\delta B \equiv \sum_{i=1}^{N}  -P\left(\frac{1}{\xi_i}\right) \Bigg[ \Big(L''(\dot x_i) & \Big)^{-1} G'\left(\frac{1}{\xi_{i+1}}\right)\frac{1}{\xi_{i+1}^2} + \Big(L''(\dot x_{i-1})\Big)^{-1}G'\left(\frac{1}{\xi_{i-1}}\right)\frac{1}{\xi_{i-1}^2} \\
& - \left(\Big(L''(\dot x_i)\Big)^{-1}+\Big(L''(\dot x_{i-1})\Big)^{-1}\right) G'\left(\frac{1}{\xi_i}\right)\frac{1}{\xi_{i}^2} \Bigg].
\end{aligned}
\]
By periodicity, we shift the indices and rewrite the preceding expression as a sum of products
\begin{equation}\label{eq:delta_B}
\begin{aligned}
\delta B \equiv 
\sum_{i=1}^{N} \Big(L''(\dot x_i)\Big)^{-1} \left( P\left(\frac{1}{\xi_{i+1}}\right) - P\left(\frac{1}{\xi_i}\right) \right) \left( G'\left(\frac{1}{\xi_{i+1}}\right)\frac{1}{\xi_{i+1}^2} - G'\left(\frac{1}{\xi_i} \right)\frac{1}{\xi_{i}^2}\right).
\end{aligned}
\end{equation}
Because $L$ is strictly convex, we have $\Big(L''(x_i)\Big)^{-1} > 0$.
Also, the convexity of $G$ implies that $G'$ is non-decreasing.
Thus, by \eqref{eq:P_nondecreasing}, $P\left(\frac{1}{\xi_{i+1}}\right) - P\left(\frac{1}{\xi_i}\right) $ and $G'\left(\frac{1}{\xi_{i+1}}\right)\frac{1}{\xi_{i+1}^2} - G'\left(\frac{1}{\xi_i} \right)\frac{1}{\xi_{i}^2}$ have the same sign for all integers $i$ such that $1 \le i \le N$, which implies the non-negativity of (B).
This completes the proof of the inequality in \eqref{eq:disp_conv_disc_U}.
\end{proof}

\begin{rem} When $U(0)=0$ (hence, $P(0)=0$ by definition), Theorem \ref{thm:disp_conv} holds for $\Omega = \R$, where we take $x_0=-\infty$ and $x_{N+1} =+\infty$.
For this case, because $\frac{1}{\xi_1} = \frac{\delta}{x_1-x_0}$ and $\frac{1}{\xi_{N+1}} = \frac{\delta}{x_{N+1}-x_{N}}$ vanish, in the preceding proof of Theorem \ref{thm:disp_conv}, \eqref{eq:delta_B} becomes
\[
\begin{aligned}
\delta B  & =  \Big(L''(\dot x_1)\Big)^{-1} P\left(\frac{1}{\xi_2}\right) G'\left(\frac{1}{\xi_2}\right)\frac{1}{\xi_2^2} \\
& + \sum_{i=2}^{N-1}\Big(L''(\dot x_i)\Big)^{-1} \left( P\left(\frac{1}{\xi_{i+1}}\right) - P\left(\frac{1}{\xi_{i}}\right) \right) \left( G'\left(\frac{1}{\xi_{i+1}} \right)\frac{1}{\xi_{i+1}^2} - G'\left(\frac{1}{\xi_{i}}\right)\frac{1}{\xi_i^2}\right) \\
& + \Big(L''(\dot x_N)\Big)^{-1} P\left(\frac{1}{\xi_N}\right) G'\left(\frac{1}{\xi_N}\right)\frac{1}{\xi_N^2}.
\end{aligned}
\]
Because $P(0)=0$ and $P(z)$ is non-decreasing, we have that $P\left(\frac{1}{\xi_2}\right) \ge 0$ and $P\left(\frac{1}{\xi_N}\right) \ge 0$.
Thus, because of the same arguments as in the proof of Theorem \ref{thm:disp_conv}, the preceding expression is non-negative.
\end{rem}

In the following proposition, we prove the log-convexity of $L^p$ norms of the density function, $R_i(t)$, at a discrete level for $p \ge 0$.

\begin{pro} Let $\xbf \in \mathcal{K}^N$ be a minimizer of \eqref{eq:global_problem}.
Then, for all $p$ such that $0 \le p < \infty$ and $t \in [0,T]$,
\begin{equation}\label{eq:log_conv}
\sum_{i=1}^N (R_i (t))^p \le \left( \sum_{i=1}^N (R_i (0))^p \right)^{1-\frac{t}{T}} \left( \sum_{i=1}^N (R_i (T))^p \right)^{\frac{t}{T}},
\end{equation}
where $R_i(t)$ is the discrete density defined by \eqref{eq:discrete_density}.
\end{pro}
\begin{proof}
The inequality \eqref{eq:log_conv} defines the logarithmic convexity of the map
\[
t \mapsto \sum_{i=1}^N (R_i (t))^p.
\]
Let $\xi_i = \frac{1}{R_i} = \frac{x_i-x_{i-1}}{\delta}$. Thus, proving \eqref{eq:log_conv} is equivalent to show that
\[
\frac{d^2}{dt^2} \ln \left( \sum_{i=1}^{N} \xi_i^{-p} \right) \ge 0,
\]
for $p \in [0,+\infty)$.
If $p=0$, both sides of the preceding inequality are equal.
For $p>0$, denote
\[
S = \sum_{i=1}^{N} \xi_i^{-p}.
\]
Because
\[
\frac{d^2}{dt^2} \ln (S) = \frac{\ddot S S - \dot S^2}{S^2},
\]
it suffices to show that
\[
\ddot S S \ge \dot S^2,
\]
or, namely,
\[
-p\left( \sum_{i=1}^N \xi_i^{-p-1} \ddot \xi_i - (p+1) \xi_i^{-p-2} \dot \xi_i^2 \right) \left( \sum_{i=1}^N \xi_i^{-p} \right) \ge (-p)^2 \left( \sum_{i=1}^N \xi_i^{-p-1} \dot \xi_i \right)^2.
\]
Dividing both sides of the previous inequality by $-p$, which is negative, we prove the reverse inequality
\begin{equation}\label{ineq_log_conv_2}
\underbrace{\left( \sum_{i=1}^N \xi_i^{-p-1} \ddot \xi_i \right)\left(\sum_{i=1}^N \xi_i^{-p}\right)}_{C} - (p+1) \left( \sum_{i=1}^N \xi_i^{-p-2} \dot \xi_i^2 \right) \left(\sum_{i=1}^N \xi_i^{-p}\right) \le -p \left( \sum_{i=1}^N \xi_i^{-p-1} \dot \xi_i \right)^2.
\end{equation}
From the proof of Theorem \ref{thm:disp_conv}, we have that the term $(B)$ in \eqref{eq:second_der_U} is non-negative.
Consider $U(z) = z^{p+1}$, which satisfies assumptions of the Theorem \ref{thm:disp_conv}.
Hence, the particular case of the term $(B)$ satisfies
\[
-p\left(\sum_{i=1}^{N} \xi_i^{-p-1}\ddot \xi_i \right)\ge 0.
\]
Because $-p<0$ and $\xi_i > 0$, the term $(C)$ in \eqref{ineq_log_conv_2} is non-positive.

Now, it suffices to prove that
\begin{equation}\label{eq:ineq_Lp_log_conv}
-(p+1) \left( \sum_{i=1}^{N} \xi_i^{-p-2} \dot \xi_i^2 \right)\left(\sum_{i=1}^{N} \xi_i^{-p}\right) \le -p \left(\sum_{i=1}^{N}\xi_i^{-p-1}\dot \xi_i\right)^2.
\end{equation}
We consider sequences $\left\{\xi_i^{\frac{-p-2}{2}}\dot \xi_i \right\}_{i=1}^{N}$ and $\left\{\xi_i^{\frac{-p}{2}}\right\}_{i=1}^{N}$ and apply Cauchy-Schwarz inequality to obtain
\[
\left( \sum_{i=1}^{N}\xi_i^{-p-2} \dot \xi_i^2 \right)\left(\sum_{i=1}^{N} \xi_i^{-p}\right) \ge\left(\sum_{i=1}^{N}\xi_i^{-p-1}\dot \xi_i\right)^2.
\]
Since $-(p+1) < -p < 0$ for $p > 0$, after multiplying the left-hand side and right-hand side of the previous inequality by $-(p+1)$ and $-p$, respectively, we get \eqref{eq:ineq_Lp_log_conv}.
\end{proof}

\section{Uniform estimates}\label{sec:estimates}
As mentioned in the Introduction, we want to approximate the continuum density $m$ in \eqref{eq:mfg} by a discrete particle density for the periodic case, $\Omega=\T$.
 This section aims to use the result in Theorem \ref{thm:disp_conv} to derive uniform estimates for such a discrete particle density.
More precisely, we set
\[m^N(x,t)=\sum_{i=0}^N R_i(t)\mathbf{1}_{I_i(t)}(x)\,,\]
with the notation
\begin{align*}
& I_i(t)=[x_i(t),x_{i+1}(t))\,,\qquad \hbox{for $i=1,\ldots,N-1$},\\
\end{align*}
and the usual convention
\begin{equation}\label{eq:convention}
x_0=x_N-1\,,\qquad x_{N+1}=x_1+1.
\end{equation}
We assume, for a given $p\in [1,+\infty]$,
\begin{equation*}
m_0,\, m_T \in L^p(\T).
\end{equation*}
For a fixed $N\in \mathbb{N}$, we consider atomizations $(x_{0,1},\ldots,x_{0,N}),\,  (x_{T,1},\ldots,x_{T,N})\in \T^N$ of the initial and terminal data, $m_0, m_T$, respectively.
More precisely, we assume
\begin{align*}
& \int_{x_{0,i}}^{x_{0,i+1}} m_0(x) dx = \delta_N\,,\qquad \int_{x_{T,i}}^{x_{T,i+1}} m_T(x) dx = \delta_N,
\end{align*}
for all $i=0,\ldots,N$ and usual convention \eqref{eq:convention}.
We set
\[m_0^N(x)=\sum_{i=0}^N R_{0,i}\mathbf{1}_{I_{0,i}}(x)\,,\qquad m_T^N(x)=\sum_{i=0}^N R_{T,i}\mathbf{1}_{I_{T,i}}(x),\]
with
\begin{align*}
& I_{0,i}=[x_{0,i},x_{0,i+1})\,,\qquad I_{T,i}=[x_{T,i},x_{T,i+1}),\\
& R_{0,i}=\frac{\delta_N}{x_{0,i+1}-x_{0,i}}\,,\qquad R_{T,i}=\frac{\delta_N}{x_{T,i+1}-x_{T,i}}\,.
\end{align*}
We use Jensen's inequality 
\begin{align*}
& \int_{\T} m_0^N(x)^p dx \leq \sum_{i=0}^N\int_{x_{0,i}}^{x_{0,i+1}} \left(\frac{\delta_N}{x_{0,i+1}-x_{0,i}}\right)^p dx \leq \sum_{i=0}^N\int_{x_{0,i}}^{x_{0,i+1}} \left(\avint_{x_{0,i}}^{x_{0,i+1}}m_0(y)dy\right)^p dx  \\
& \ \leq \sum_{i=0}^N\int_{x_{0,i}}^{x_{0,i+1}}\avint_{x_{0,i}}^{x_{0,i+1}}m_0(y)^p dy dx = \int_\T m_0(y)^p dy,
\end{align*}
and a similar computation holds for $m_T$.
Therefore, $m_0$ and $m_T$ are uniformly bounded in $L^p$ with respect to $N$.
The case $p=+\infty$ easily follows by sending $p\rightarrow +\infty$.
Consequently, Theorem \ref{thm:disp_conv} with $U(z)=z^p$ implies
\[\sup_{t\in [0,T]\,,N\in \mathbb{N}} \|m^N(\cdot,t)\|_{L^p(\T)}<+\infty.\]

\section{Numerical tests}\label{numerics}
In this section, we present numerical simulations for the particle approximation of Problem \ref{prob:mfg}.
We have that the solution, $\xbf \in \mathcal{K}^N$, of Problem \eqref{prob:global} represents the quantiles of the distribution function of $m(x,t)$, which solves Problem \ref{prob:mfg}.
Thus, here, we approximate the CDF of $m$ by discretizing \eqref{eq:global_problem} in time and experimentally show that some quantities from Section \ref{sec:cons_quan} are conserved.
We also numerically verify that the displacement convexity property proven in Theorem \ref{thm:disp_conv} holds.

We split the interval $[0,T]$ into $N_T$ subintervals of equal size $\Delta t = \frac{T}{N_T}$.
Let $t_n = n \Delta t$ and denote $x_{i}^{n} = x_i(t_n)$ for an integer $n$ such that $0 \le n \le N_T$.
We approximate $\dot x_{i}^n$ by the forward difference $\frac{x_i^{n+1}-x_i^{n}}{\Delta t}$ and discretize the time integral in $\widetilde \J$ given by \eqref{eq:global_problem}.
At a discrete level, we minimize the utility functional
\begin{align}
\label{eq:discrete_min_problem}
\frac{1}{N} & \sum_{i=1}^{N} \sum_{n=0}^{N_T-1} L \left( \frac{x_i^{n+1}-x_i^{n}}{\Delta t} \right) + \frac{1}{N}\sum_{i=1}^{N} \sum_{n=0}^{N_T-1} G\left(\frac{\delta}{x_{i}^n-x_{i-1}^n}\right) - \frac{1}{N}\sum_{i=1}^N  \sum_{n=0}^{N_T-1} V(x_i^n,t_n)
\end{align}
over all $x_i^n$ with $i,n$ such that $1 \le i \le N$ and $1 \le n \le N_T-1$.
Here, we drop the multiplicative constant coefficient $\Delta t$ because it does not affect the solution. 
We seek minimizers $\bar{x}_i^n$ for \eqref{eq:discrete_min_problem}, where the indices 
$i$ and $n$ are integers such that $1 \le i \le N$ and $1 \le n \le N_T-1$.
In all of the following tests, we take $H(p)=\frac{p^2}{2}$, from which we determine that $L(v)=\frac{v^2}{2}$, using the Legendre transform.
To validate our numerical experiments, we compare our numerical results
with non-trivial solutions of Problem \ref{prob:mfg}. For that, 
 we use the following procedure to generate solutions of Problem \ref{prob:mfg}. 
First, we choose a density function $m(x,t)$. 
Then, from the Fokker-Planck equation of \eqref{eq:mfg}, we determine $u(x,t)$. Finally, 
we use the Hamilton-Jacobi equation in \eqref{eq:mfg} to determine a potential 
 $V(x,t)$ for which $(u,m)$ solves Problem \ref{prob:mfg}. While solving Problem \ref{prob:mfg}
 for arbitrary $V$ is often an impossible task.
 The method just described is extremely useful in generating non-trivial solutions. 

For the minimization of \eqref{eq:discrete_min_problem}, we take the initial, $\xbf^0 = (x_1(0),\dots,x_N(0)) \in \mathcal{K}^N$, and terminal, $\xbf^T = (x_1(T),\dots,x_N(T)) \in \mathcal{K}^N$, states of agents to be the quantiles of the CDF of the density function, $m(x,t)$, for $t=0$ and $t=T$, respectively.
We take linear interpolations between each agent's initial and final states as the initial guess of the trajectories. 
Finally, we minimize \eqref{eq:discrete_min_problem} directly and plot the approximate CDF by $F(\bar x_i^n,t_n) = \frac{i}{N}$ for a given $t_n$.
In all the numerical tests below, except when stated otherwise, we consider $T=1$ and $N_T=100$.

\subsection{Test 1}
For the first numerical example, the domain is the unit torus $\Omega = \T$ identified with $\R/\Z$ and $g(m)=m^2/2$.
Thus, by \eqref{eq:enthalpy}, $G(r)=r^2/6$.
We consider $m(x,t) = 1+ \sin(2\pi x)e^{-t-1}$ in $\T \times [0,T]$.
Let $\{x\}$ denote the fractional part of $x$.
Using \eqref{eq:cdf} for $\Omega=\T$, we find that the corresponding CDF is
\[
\varphi(x,t) = \{x\}+\frac{e^{-1-t} \sin^2(\pi x)}{\pi}.
\]
From the second equation of the system \eqref{eq:mfg}, we obtain
\[
u(x,t) = \frac{\log(e^{1+t}+\sin(2\pi x))}{4\pi^2}.
\]
Also, from the first part of \eqref{eq:mfg}, we deduce that
\begin{equation}\label{eq:V_periodic}
V(x,t) = \frac{1}{8} \left(4+\frac{1}{\pi^2}+8e^{-1-t}\sin(2\pi x)+4e^{-2(1+t)}\sin^2(2\pi x)+\frac{-1+e^{2+2t}}{\pi^2(e^{1+t}+\sin^2(2\pi x))^2}\right).
\end{equation}
Hence, we find a minimum point, $\bar x_i^n$, of \eqref{eq:discrete_min_problem} for all integers $i,n,$ such that $1 \le i \le N$ and $1 \le n \le N_T-1$ and approximate the CDF by $F(\bar x_i^n,t_n) = \frac{i}{N}$.
Figure \ref{fig:periodic_case} (A) illustrates the optimal trajectories of the particles minimizing \eqref{eq:discrete_min_problem}.
In Figure \ref{fig:periodic_case} (B), the exact and approximate CDFs for $m(x,T/2)$ are displayed for $N=50$, and we see that our approximation fits this periodic case.
\begin{figure}
\centering
	\begin{tabular}{ccc}
	\includegraphics[width=50mm]{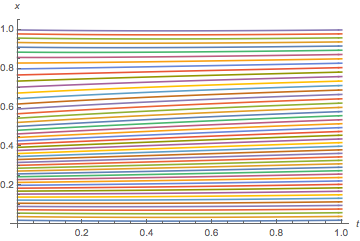}&
	\includegraphics[width=50mm]{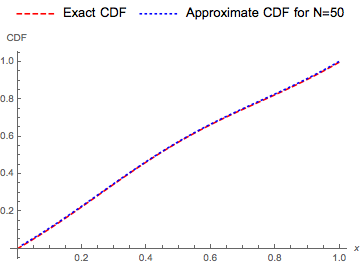}&  \\
	(A) & (B)
	\end{tabular}
	\caption{Periodic case: $\Omega=\T$, $g(m)=m^2/2$ (hence, $G(r)=r^2/6$) and $V(x,t)$ is given by \eqref{eq:V_periodic}.
(A) Optimal trajectories of the $N=50$ particles minimizing \eqref{eq:discrete_min_problem}.
(B) Exact and approximate CDFs for $N=50$ at $t=T/2$}\label{fig:periodic_case}
\end{figure}

\subsection{Test 2}
For $\Omega = \R$, we consider $g(m) = m$ in \eqref{eq:mfg}, so that $G(r) = \frac{r}{2}$ by \eqref{eq:enthalpy}.
We take the density function
\begin{equation}\label{eq:density_test1}
m(x,t) = \frac{1}{\pi \left( 1+\left( t+ \frac{t^2}{20} - x\right)^2 \right)}
\end{equation}
as a solution to \eqref{eq:mfg}.
Then,
\[
u(x,t) = -x-\frac{xt}{10}.
\]
From the Hamilton-Jacobi equation of the MFG \eqref{eq:mfg},
\begin{equation}\label{eq:V_test2}
V(x,t) = -\frac{1}{200}\left(10+t\right)^2+\frac{1}{\pi \left(1+\left(t+\frac{t^2}{20}-x\right)^2\right)} - \frac{x}{10}.
\end{equation}

The corresponding CDF for the density $m(x,t)$ in \eqref{eq:density_test1} is given by
\begin{equation*}
\varphi(x,t) =\frac{1}{2}-\frac{\arctan \left(\frac{t^2}{20}+t-x \right)}{\pi}.
\end{equation*}
Figure \ref{fig:realline_case} displays the optimal trajectories and exact CDF with its numerical approximation for $N=50$.

\begin{figure}
\centering
	\begin{tabular}{ccc}
	\includegraphics[width=50mm]{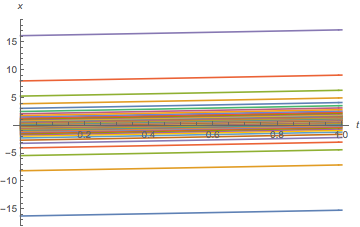}&
	\includegraphics[width=50mm]{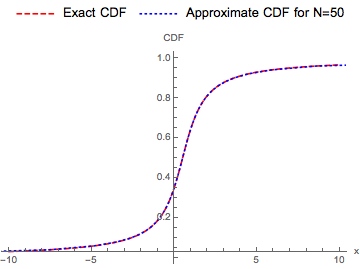}&  \\
	(A) & (B)
	\end{tabular}
	\caption{$\Omega=\R$, $g(m)=m$ (hence, $G(r)=r/2$) and $V(x,t)$ is given by \eqref{eq:V_test2}.
(A) Optimal trajectories of the $N=50$ particles minimizing \eqref{eq:discrete_min_problem}.
(B) Exact and approximate CDFs for $N=50$ at $t=T/2$}\label{fig:realline_case}
\end{figure}


\subsection{Test 3}
For $\Omega = \R$, we consider $g(m) = m$ again so that $G(r) = \frac{r}{2}$.
Here, we take a different density function, namely,
\[
m(x,t) = \frac{1}{2+2 \cosh{(t+t^3-x)}}.
\]
Consequently, $u(x,t) = -(1+3t^2)x$ and
\begin{equation}\label{eq:V_test3}
V(x,t) = -\frac{1}{2}(1+3t^2)^2-6tx+\frac{1}{2+2\cosh{(t+t^3-x)}}.
\end{equation}
The CDF for the density function, $m(x,t)$ is
\[
\varphi(x,t) = \frac{1}{1+e^{t+t^3-x}}.
\]
Figure \ref{fig:realline_case2} (A) illustrates the optimal trajectories between the initial and terminal positions.
In Figure \ref{fig:realline_case2} (B), the CDF and approximate CDF obtained from the minimization problem \eqref{eq:discrete_min_problem} are plotted.

\begin{figure}
\centering
	\begin{tabular}{ccc}
	\includegraphics[width=50mm]{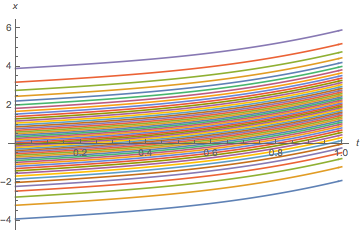}&
	\includegraphics[width=50mm]{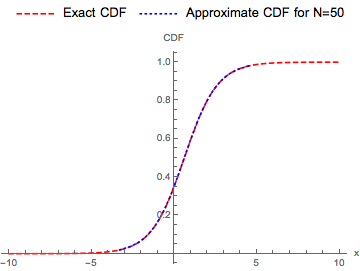}&  \\
	(A) & (B)
	\end{tabular}
	\caption{$\Omega=\R$, $g(m)=m$ (hence, $G(r)=r/2$) and $V(x,t)$ is given by \eqref{eq:V_test3}.
(A) Optimal trajectories of the $N=50$ particles minimizing \eqref{eq:discrete_min_problem}.
(B) Exact and approximate CDFs for $N=50$ at $t=T/2$}\label{fig:realline_case2}
\end{figure}
 From Figures \ref{fig:periodic_case}-\ref{fig:realline_case2}, we see that CDFs of the density functions, $m$, that solve Problem \ref{prob:mfg} are well-approximated.
Moreover,  it is consistent with the corresponding optimal trajectories.
For example, in Figure \ref{fig:realline_case}, the optimal trajectories are more concentrated in the middle.
Consequently, the CDF also considerably increases in the middle and becomes more stable.
 
\subsection{Discrete conserved quantities and displacement convexity}
Here, we test if semi-discrete conserved quantities identified in Section \ref{sec:cons_quan} are preserved at a fully discrete level. In particular, we study how the time-discretization of the semi-discrete quantities affects their conservation over the space domain.

Differentiating \eqref{eq:discrete_min_problem} (multiplied by $N$) w.r.t $x_j^k$, where $2 \le j \le N-1$ and $1 \le k \le N_T-1$, we get
\[
\begin{aligned}
0 = & \frac{L'\left(\frac{x_j^k-x_j^{k-1}}{\Delta t}\right) - L'\left(\frac{x_j^{k+1}-x_j^{k}}{\Delta t}\right)}{\Delta t}\\
& +G'\left(\frac{\delta}{x_{j+1}^k-x_{j}^k}\right)\frac{\delta}{\left(x_{j+1}^k-x_{j}^k\right)^2} - G'\left(\frac{\delta}{x_{j}^k-x_{j-1}^m}\right)\frac{\delta}{\left(x_{j}^k-x_{j-1}^k\right)^2} - V_x(x_j^k,t_k).
\end{aligned}
\]
Hence, for a constant $V$, by the telescopic sum
\[
\sum_{i=1}^N \frac{L'\left(\frac{x_i^{k+1}-x_i^{k}}{\Delta t}\right) - L'\left(\frac{x_i^k-x_i^{k-1}}{\Delta t}\right)}{\Delta t} = 0.
\]
This implies that the value of the sum $\sum_{i=1}^N L'\left(\frac{x_i^{k+1}-x_i^{k}}{\Delta t}\right)$ is the same for all $k$ such that $0 \le k \le N_T-1$.
Hence, the semi-discrete conserved quantity \eqref{eq:cons_momentum} is preserved at a discrete level.
To illustrate this result numerically, we consider the case $L(v) = v^2/2$ and $G(r)=r^2/6$, randomly generate and fix initial-terminal conditions, $\xbf^0 \in \mathcal{K}$ and $\xbf^T \in \mathcal{K}$ for $N=5$.
Taking $N_T=20$, we find a minimum of \eqref{eq:discrete_min_problem} and plot the obtained optimal trajectories of the particles in Figure \ref{fig:disc_cons_quan_1} (A).
Also, in Figure \ref{fig:disc_cons_quan_1} (B), the values of $\sum_{i=1}^N L'\left(\frac{x_i^{k+1}-x_i^{k}}{\Delta t}\right)$ for all $t_k=k \Delta t$, where $0 \le k \le N_T-1$, are provided.
We see that it is constant in time.
However, Figure \ref{fig:disc_cons_quan_1} reveals that the semi-discrete conserved quantity given by \eqref{eq:cons_energy} is not perfectly conserved. However, as $\Delta t \to 0$, we see that the curve gets flatter except at initial and terminal times, where there seems to appear a boundary layer. In Figure \ref{fig:disp_conv}, we plot \eqref{eq:disp_conv} with $U(z) = e^{-z}$. 
This figure shows that Theorem \ref{thm:disp_conv} about the displacement convexity proven in Section \ref{sec:disp_conv} also holds as expected.
\begin{figure}
\begin{center}
	\begin{tabular}{ccc}
	\includegraphics[width=50mm]{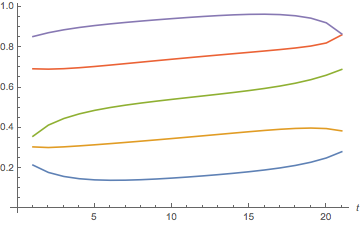}&
	\includegraphics[width=50mm]{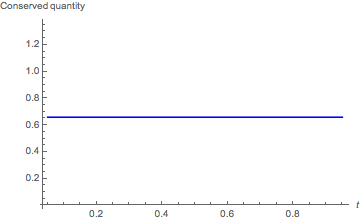} & \\
	(A) & (B)
	\end{tabular}
	\includegraphics[width=50mm]{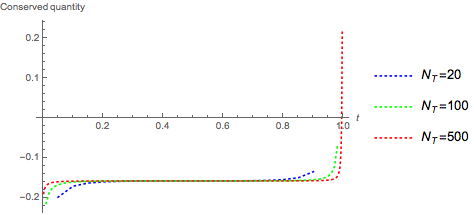} \\
	(C)
	
	\caption{$L(v)=v^2/2$.
(A) Optimal trajectories of the particles.
(B) Discrete conserved quantity $\sum_{i=1}^{N} L'\left(\frac{x_i^{m+1}-x_i^m}{\Delta t}\right)R_i$.
(C) Discrete approximation of the semi-discrete quantity $\sum_{i=1}^N (L'(u_i)u_i-L(u_i)-G(R_i))R_i$ given by \eqref{eq:cons_energy}}\label{fig:disc_cons_quan_1}
\end{center}
\end{figure}

\begin{figure}
\centering
	\includegraphics[scale=0.6]{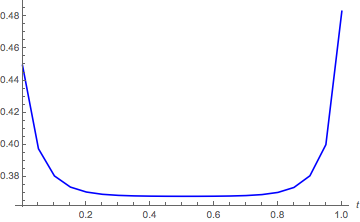}
	\caption{Displacement convexity illustration for $\sum_{i=1}^{N} U(R_i(t)) (x_i(t) - x_{i-1}(t))$ with $U(z)=e^{-z}$, $N=5, \Delta t = \frac{1}{20}$} \label{fig:disp_conv}
\end{figure}

\section{Conclusion}
In this paper, we considered a numerical approximation of the CDF of the density function, $m$, that solves Problem \ref{prob:mfg}. 
We studied an equivalent minimization Problem \ref{prob:global}, for which we also showed existence and uniqueness for convex $V$. 
The numerical experiments show that the proposed particle method provides a good approximation for the CDF. 
However, further error estimates need to be further developed.
In addition, our method preserves the displacement convexity property that holds for the continuous case and was proven in \cite{gomes2018displacement} and some of the conserved quantities identified in \cite{MR3575617}. In particular, displacement convexity yields that particles will never cross.

In future work, we may consider the convergence of the approximate solution to the exact CDF as $N \to \infty$ and error analysis. Because Theorem \ref{thm:disp_conv} implies uniform bounds as was discussed in Section \ref{sec:estimates}, the result may be useful in the proof of the convergence.

\bibliographystyle{plain}

\IfFileExists{"/Users/ribeirrd/Dropbox/MasterBIB/mfg.bib"}
{\bibliography{/Users/ribeirrd/Dropbox/MasterBIB/mfg.bib}}
{\bibliography{mfg.bib}} 


\end{document}